\documentclass[11pt,reqno]{amsart}
\usepackage{hyperref}
\usepackage{amssymb,amsmath,amsfonts,latexsym}
\usepackage{bm}
\usepackage{mathtools}
\usepackage{enumerate}
\usepackage{enumitem}
\usepackage{mathrsfs}
\usepackage{array,graphics,color}
\usepackage{csquotes}
\usepackage[utf8]{inputenc}
\allowdisplaybreaks

\numberwithin{equation}{section}
\newtheorem{dfn}{Definition}[section]
\newtheorem{thm}[dfn]{Theorem}
\newtheorem{lma}[dfn]{Lemma}

\newtheorem{ppsn}[dfn]{Proposition}
\newtheorem{crlre}[dfn]{Corollary}
\newtheorem{xmpl}[dfn]{Example}
\newtheorem{rmrk}[dfn]{Remark}

\newtheorem{note}[dfn]{Note}

\DeclarePairedDelimiterX{\norm}[1]{\lVert}{\rVert}{#1}
\DeclarePairedDelimiterX{\bnorm}[1]{\big\lVert}{\big\rVert}{#1}
\DeclarePairedDelimiterX{\Bnorm}[1]{\Big\lVert}{\Big\rVert}{#1}

\newcommand{\N}{\mathbb{N}}

\begin{document}
	
	\title[On a generalized density point defined by families of sequences involving ideals]{On a generalized density point defined by families of sequences involving ideals}
	
	\author[Banerjee] {Amar Kumar Banerjee}
	\address{Department of Mathematics, The University of Burdwan, Burdwan-713104, West Bengal, India}
	\email{akbanerjee1971@gmail.com, akbanerjee@math.buruniv.ac.in}
	
	\author[Debnath] {Indrajit Debnath}
	\address{Department of Mathematics, The University of Burdwan, Burdwan-713104, West Bengal, India}
	\email{ind31math@gmail.com}

	\subjclass[2020]{40A35, 54C30, 26E99}
	
	\keywords{Density topology, ideal, $\mathcal{I}$-density topology}
	
	\begin{abstract}
		In this paper we have introduced the notion of $\mathcal{I}_{(s)}$-density point corresponding to the family of unbounded and $\mathcal{I}$-monotonic increasing positive real sequences, where $\mathcal{I}$ is the ideal of subsets of the set of natural numbers. We have studied the corresponding topology in the space of reals and have investigated several properties of this topology. Also we have formulated a weaker condition for the sequences so that the classical density topology coincides with $\mathcal{I}_{(s)}$-density topology.
	\end{abstract}
	\maketitle
	
	\section{Introduction}
	
	A series of important developments in density topology were evolved from the foundational result of Goffman et al. \cite{Goffman} to the most remarkable work of M. Filipczak and J. Hejduk in \cite{Filipczak 2004} where they defined the density point by families of sequences. Density topology were studied extensively in several spaces like the space of real numbers \cite{Riesz}, Euclidean $n$-space \cite{Troyer}, metric spaces \cite{Lahiri 1998} etc. In the recent past the notion of classical Lebesgue density point were generalized by weakening the assumptions on the sequences of intervals and consequently several notions like $\langle s \rangle$-density point by M. Filipczak and J. Hejduk \cite{Filipczak 2004}, $\mathcal{J}$-density point by J. Hejduk and R. Wiertelak  \cite{Hejduk 2014}, $\mathcal{S}$-density point by F. Strobin and R. Wiertelak \cite{Strobin} were obtained. A significant volume of work in this area were carried out by distinguished researchers in the last few decades \cite{das, Filipczak 2021, Hejduk 2018, Wojdowski}. In recent time Banerjee and Debnath have found a new way to generalize density topology using ideals in \cite{banerjee 4}. 
	
	The usual notion of convergence doesnot always capture the properties of vast class of non-convergent sequences in fine details. In order to include more sequences under purview the idea of convergence of real sequences was generalized to the notion of statistical convergence \cite{Fast,Schoenberg} followed by the idea of ideal convergence \cite{Kostyrko 2000}.

	$\langle s \rangle$-density topology \cite{Filipczak 2004} is the object of our interest and play a central role in our study. The prime objective of this paper is to investigate a generalized density point defined by families of sequences. In this paper we try to generalize the $\langle s \rangle$-density point by involving the notion of ideal $\mathcal{I}$ of subsets of naturals. We have given the notion of $\mathcal{I}_{(s)}$-density and induced $\mathcal{I}_{(s)}$-density topology in the space of reals. We have shown that $\mathcal{I}_{(s)}$-density point is dependent on the nature of the sequence $(s)$. Some natural properties of this topology have been studied. Also we have given a characterization of equality between this topology and classical density topology. 
	
	\section{Preliminaries}
	Let us recall the definition of asymptotic density. Here $\mathbb{N}$ stands for the set of natural numbers and for $K \subset \mathbb{N}$
	we denote $K(n)$ to be the set $\{k \in K : k \leq n\}$ and $|K(n)|$ is the cardinality of $K(n)$. The asymptotic density of $K$ is defined by  $d(K)=\lim_{n\rightarrow{\infty}}\frac{|K(n)|}{n}$, provided the limit exists. The notion of asymptotic density was used to define the idea of statistical convergence by Fast \cite{Fast}, generalizing the idea of usual convergence of real sequences. A sequence $\{x_n\}_{n \in \mathbb{N}}$ of real numbers is said to be statistically convergent to $x_0$ if for given any $\epsilon > 0$ the set $K(\epsilon) = \{k\in \mathbb{N} : |x_k - x_0| \geq \epsilon\}$ has asymptotic density zero.

	After this pioneering work, the theory of statistical convergence of real sequences were generalized to the idea of $\mathcal{I}$-convergence of real sequences by P. Kostyrko et al. \cite{Kostyrko 2000}, using the notion of ideal  $\mathcal{I}$ of subsets of $\mathbb{N}$, the set of natural numbers. We shall use the notation $2^{\mathbb{N}}$ to denote the power set of $\mathbb{N}$.
	
	\begin{dfn}\cite{Kostyrko 2000}
		A nonvoid class  $\mathcal{I} \subset 2^\mathbb{N}$ is called an ideal if $A,B \in  \mathcal{I}$ implies $A \cup B \in  \mathcal{I}$ and $A\in  \mathcal{I}, B\subset A$ imply $B\in \mathcal{I}$. Clearly $\{\phi\}$ and $2^{\mathbb{N}}$ are ideals of $\mathbb{N}$ which are called trivial ideals. An ideal is called non-trivial if it is not trivial.
	\end{dfn}

	It is easy to verify that the family $\mathcal{J}=\{A \subset \mathbb{N}: d(A)=0\}$ forms a non-trivial admissible ideal of subsets of $\mathbb{N}$. If $\mathcal{I}$ is a proper non-trivial ideal, then the family of sets $\{M\subset \mathbb{N} : \mathbb{N}\setminus M \in \mathcal{I}\}$ denoted by $\mathcal{F}(\mathcal{I})$ is a filter on $\mathbb{N}$ and it is called the filter associated with the ideal $\mathcal{I}$ of $\mathbb{N}$.

	\begin{dfn}\cite{Kostyrko 2000}
		A sequence $\{x_n\}_{n \in \mathbb{N}}$ of real numbers is said to be $\mathcal{I}$-convergent to $x_0$ if the set $K(\epsilon) = \{k\in \mathbb{N} : |x_k - x_0| \geq \epsilon\}$ belongs to $\mathcal{I}$ for any $\epsilon>0$.
	\end{dfn}
	
	Further many works were carried out in this direction by many authors \cite{banerjee 21,banerjee 31,Lahiri 2005}. Throughout the paper the ideal $\mathcal{I}$ will always stand for a nontrivial admissible ideal of subsets of $\mathbb{N}$.

	Now let us introduce the following notations which will serve our purpose. Throughout $\mathbb{R}$ stands for the set of all real numbers. We shall use the notation $\mathcal{L}$ for the $\sigma$-algebra of Lebesgue measurable sets on $\mathbb{R}$, $\lambda ^{\star}$ for the outer Lebesgue measure and $\lambda$ for the Lebesgue measure on $\mathbb{R}$ \cite{halmos}. Wherever we write $\mathbb{R}$ it means that $\mathbb{R}$ is equipped with natural topology unless otherwise stated. We shall use the notation $2^{\mathbb{R}}$ to denote the power set of $\mathbb{R}$. By \enquote{Euclidean $F_{\sigma}$ and Euclidean $G_{\delta}$ set} we mean $F_{\sigma}$ and $G_{\delta}$ set in $\mathbb{R}$ equipped with natural topology. The symmetric difference of two sets $A$ and $B$ is $(A  \setminus B)\cup (B \setminus A)$ and it is denoted by $A \triangle B$. The fact that $\lambda(A \triangle B)=0$ for any two sets $A, B \in \mathcal{L}$ will be denoted by $A \sim B$. By \enquote{a sequence of intervals $\{J_n\}_{n \in \mathbb{N}}$ about a point $p$} we mean $p \in \bigcap_{n \in \mathbb{N}}J_n$. The length of the interval $J_n$ will be denoted by $|J_n|$.
	
	\begin{dfn}\cite{Density topologies} For $E \in \mathcal{L}$ and a point $p \in \mathbb{R}$ we say the point $p$ is a classical density point of $E$ if and only if  $$\lim_{h \rightarrow{0+}} \frac{\lambda(E \cap [p-h,p+h])}{2h}=1.$$   
	\end{dfn}
	
	Equivalently we can say the point $p \in \mathbb{R}$ is a classical density point of $E$ if and only if $$\lim_{h \rightarrow{0+}} \frac{\lambda((\mathbb{R} \setminus E) \cap [p-h,p+h])}{2h}=0.$$
	
	The set of all classical density point of $E$ is denoted by $\Phi (E)$. The collection $$\mathcal{T}_d=\{E \in \mathcal{L}: E \subseteq \Phi(E) \}$$ is a topology in the real line \cite{Density topologies} and it is called as the classical density topology.
	
	\begin{thm}\cite{Oxtoby} \label{e5}
		For any Lebesgue measurable set $H \subset \mathbb{R}$,	$$\lambda(H \triangle \Phi(H))=0.$$
		
	\end{thm}

	The above theorem is known as Lebesgue Density Theorem.

	\begin{dfn} \cite{Hejduk 2012}
		We shall say that an operator $\Phi : \mathcal{L} \rightarrow{\mathcal{L}}$ is a lower density
		operator if the following conditions are satisfied:
		\begin{enumerate}
			\item $\Phi (\emptyset) = \emptyset, \Phi(\mathbb{R}) = \mathbb{R}$;
			\item $\forall A,B\in \mathcal{L},  \Phi(A \cap B) = \Phi(A) \cap \Phi(B)$;
			\item $\forall A,B \in \mathcal{L}, A \sim B \implies \Phi(A) = \Phi(B)$;
			\item $\forall A\in \mathcal{L},  A \sim \Phi(A)$.
		\end{enumerate}
		
	\end{dfn}
	
	\begin{dfn} \cite{Hejduk 2012}
		We shall say that an operator $\Psi : \mathcal{L} \rightarrow{2^{\mathbb{R}}}$ is an almost density
		operator if the following conditions are satisfied:
		\begin{enumerate}
			\item $\Psi (\emptyset) = \emptyset, \Psi(\mathbb{R}) = \mathbb{R}$;
			\item $\forall A,B\in \mathcal{L},  \Psi(A \cap B) = \Psi(A) \cap \Psi(B)$;
			\item $\forall A,B \in \mathcal{L}, A \sim B \implies \Psi(A) = \Psi(B)$;
			\item $\forall A\in \mathcal{L},  \lambda(\Psi(A) \setminus A)=0$.
		\end{enumerate}
		
	\end{dfn}
	
	\begin{rmrk} \label{e1}
		A lower density operator is an almost density operator but not conversely. For an example of an almost density operator that is not a lower density operator see \cite{Hejduk 2012}.
	\end{rmrk}
	
	\begin{thm}\cite{Hejduk 2012} \label{e2}
		Let $\Psi : \mathcal{L} \rightarrow{2^{\mathbb{R}}}$ is an almost density operator. Then the family $\mathcal{T}_{\Psi}=\{B \in \mathcal{L}: B \subseteq \Psi(B)\}$ forms a topology on $\mathbb{R}$.
	\end{thm}

	In \cite{Filipczak 2004} M. Filipczak and J. Hejduk introduced the notion of $\langle s \rangle$-density as follows. Let $\mathcal{S}$ be the family of all unbounded and non-decreasing sequence of positive reals. Every sequence $\{s_{n}\}\in \mathcal{S}$ is denoted by $\langle s \rangle$. Then a new kind of density point is defined.

	\begin{dfn} \cite{Filipczak 2004} \label{e3}
		Let $\langle s \rangle \in \mathcal{S}$. We say that $x \in \mathbb{R}$ is a density point of a set $A \in \mathcal{L}$ with respect to a sequence $\langle s \rangle \in \mathcal{S}$ or an $\langle s \rangle$-density point of $A$ if  $\lim_{n \rightarrow{\infty}} \frac{\lambda\left(A \cap \left[x-\frac{1}{s_n},x+\frac{1}{s_n}\right]\right)}{\frac{2}{s_n}}=1$.
	\end{dfn}
	
	$x$ is an $\langle s \rangle$-dispersion point of $A$ if $x$ is an $\langle s \rangle$-density point of $\mathbb{R} \setminus A$. 
	
	\begin{ppsn} \cite{Ciesielski} \label{e4}
		Let $A \in \mathcal{L}$ and $x \in \mathbb{R}$. Then
		
		$\lim_{h \rightarrow{0+}} \frac{\lambda(A \cap [x-h,x+h])}{2h}=1$ if and only if $\lim_{n \rightarrow{\infty}} \frac{\lambda\left(A \cap \left[x-\frac{1}{n},x+\frac{1}{n}\right]\right)}{\frac{2}{n}}=1$.
	\end{ppsn} 
	
	So if we choose in particular $s_n=n$ for all $n \in \mathbb{N}$ in Definition \ref{e3} then we obtain the notion of classical density point.
	
	For any sequence $\langle s \rangle \in \mathcal{S}$ and set $A \in \mathcal{L}$ let 
	\begin{equation*}
		\Phi_{\langle s \rangle}(A)=\{x \in \mathbb{R}: x\  \mbox{is} \     \langle s \rangle- \mbox{density point of A} \}.
	\end{equation*}
	
	\begin{ppsn} \cite{Filipczak 2004}
		For every pair of Lebesgue measurable sets $A,B \in \mathcal{L}$ and a sequence $\langle s \rangle \in \mathcal{S}$ we have
	\end{ppsn}
	\begin{enumerate}
		\item $\Phi_{\langle s \rangle}(\emptyset)=\emptyset$ , $\Phi_{\langle s \rangle}(\mathbb{R})=\mathbb{R}$;
		\item $\Phi_{\langle s \rangle}(A \cap B)=\Phi_{\langle s \rangle}(A) \cap \Phi_{\langle s \rangle}(B)$;
		\item $A \sim B \implies \Phi_{\langle s \rangle}(A)=\Phi_{\langle s \rangle}(B)$;
		\item $\Phi_{\langle s \rangle}(A) \sim A$;
		\item $\Phi(A) \subseteq \Phi_{\langle s \rangle}(A). $

	\end{enumerate}
	
	\begin{crlre}\cite{Filipczak 2004}
		The operator $\Phi_{\langle s \rangle}: \mathcal{L} \rightarrow{\mathcal{L}}$ is a lower density operator in the measure space $(\mathbb{R}, \mathcal{L}, \lambda)$.
	\end{crlre}
	
	\begin{thm}\cite{Filipczak 2004}
		For every sequence $\langle s \rangle \in \mathcal{S}$ the family
		\begin{equation*}
			\mathcal{T}_{\langle s \rangle}=\{A \in \mathcal{L}: A \subseteq \Phi_{\langle s \rangle}(A)\}
		\end{equation*}
		forms a topology such that $\mathcal{T}_{d} \subseteq \mathcal{T}_{\langle s \rangle}$ and $\mathcal{T}_{\langle s \rangle}$ is the von Neumann topology associated with the Lebesgue measure.
	\end{thm}
	
	In \cite{Filipczak 2004} a characterization of equality was formulated for $\mathcal{T}_{d}$ and $\mathcal{T}_{\langle s \rangle}$.
	\begin{thm}\cite{Filipczak 2004}
		Let $\langle s \rangle \in \mathcal{S}$ be a sequence then $\mathcal{T}_{d} = \mathcal{T}_{\langle s \rangle}$ if and only if $\liminf_{n \rightarrow{\infty}} \frac{s_n}{s_{n+1}}>0$.
	\end{thm}

	\section{$\mathcal{I}_{(s)}$-density}\label{section 2}

	Through out we consider the measure space $(\mathbb{R}, \mathcal{L}, \lambda)$ where $\mathbb{R}$ is the set of real numbers, $\mathcal{L}$ is the $\sigma$-algebra of Lebesgue measurable sets and $\lambda$ is the Lebesgue measure. 
	
	\begin{dfn}\cite{salat}
		A real valued sequence $x=\{x_n\}_{n \in \mathbb{N}}$ is said to be $\mathcal{I}$-monotonic increasing (res. $\mathcal{I}$-monotonic decreasing), if there is a set $\{k_1<k_2< \dots\}\in \mathcal{F}(\mathcal{I})$ such that $x_{k_i} \leq x_{k_{i+1}} (\mbox{res.} \ x_{k_i} \geq x_{k_{i+1}})$ for every $i \in \mathbb{N}$. 
	\end{dfn}

	For any nontrivial admissible ideal $\mathcal{I}$, the family of all unbounded, $\mathcal{I}$-monotonic increasing positive real sequences is denoted by $\Sigma_{\mathcal{I}}$. If a sequence $\{s_n\}_{n \in \mathbb{N}}$ is chosen from the family $\Sigma_{\mathcal{I}}$ it will be denoted by $(s)$.

	\begin{dfn} \label{e6}
		Consider the measure space $(\mathbb{R}, \mathcal{L}, \lambda)$ and let $\mathcal{I}$ be a nontrivial admissible ideal of subsets of $\mathbb{N}$. Now for a Lebesgue measurable set $A$, a point $p$ in $\mathbb{R}$ and $(s) \in \Sigma_{\mathcal{I}}$ let us take a collection of closed intervals about $p$ as $J_{n}=\left[p-\frac{1}{s_n},p+\frac{1}{s_n}\right]$ for $n \in \mathbb{N}$.

		Now let us take $x_n = \frac{\lambda(A \cap J_n)}{|J_n|}$. Then clearly $x=\{x_n\}_{n \in \mathbb{N}}$ is a sequence of non-negative real numbers. If $\mathcal{I}-\lim_n x_n$ exists then we denote the value by $\mathcal{I}_{(s)}-d(p,A)$ which we call as $\mathcal{I}_{(s)}$-density of $A$ at the point $p$ and clearly $\mathcal{I}_{(s)}-d(p,A)=\mathcal{I}-\lim x_n$.
	\end{dfn}

	A point $p_0 \in \mathbb{R}$ is an $\mathcal{I}_{(s)}$-density point of $A \in \mathcal{L}$ if $\mathcal{I}_{(s)}-d(p_0,A)=1$. 
	
	If a point $p_0\in \mathbb{R}$ is an $\mathcal{I}_{(s)}$-density point of the set $\mathbb{R}\setminus A$, then $p_0$ is an $\mathcal{I}_{(s)}$-dispersion point of $A$.
	
	If in the above definition we take $J_{n}=\left[p,p+\frac{1}{s_n}\right]$ for $n \in \mathbb{N}$ so that $\mathcal{I}-\lim_n x_n=1$ then the point $p \in \mathbb{R}$ is a right $\mathcal{I}_{(s)}$-density point of $A$ and if we take $J_{n}=\left[p-\frac{1}{s_n},p\right]$ for $n \in \mathbb{N}$ so that $\mathcal{I}-\lim_n x_n=1$ then the point $p \in \mathbb{R}$ is a left $\mathcal{I}_{(s)}$-density point of $A$. We note that
	$$ \lambda\left(A \cap \left[p-\frac{1}{s_n},p+\frac{1}{s_n}\right]\right)=\lambda\left(A \cap \left[p-\frac{1}{s_n},p\right]\right)+\lambda\left(A \cap \left[p,p+\frac{1}{s_n}\right]\right).$$
	It can be easily proved that if a point $p \in \mathbb{R}$ is both left and right $\mathcal{I}_{(s)}$-density point of $A$, then the point $p$ is an $\mathcal{I}_{(s)}$-density point of $A$.

	\begin{note}
		Thus the following three kinds of density point may be distinguished in this context. 
	\end{note}
	\begin{enumerate}
		\item If we take the sequence $s_n=n$ for all $n \in \mathbb{N}$ then by Proposition \ref{e4} we obtain the notion of classical density point. 
		\item If $\{s_n\} \in \mathcal{S}$ then the notion of $\langle s \rangle$-density point is obtained.
		\item For $\{s_n\} \in \Sigma_{\mathcal{I}}$ we have introduced the notion of $\mathcal{I}_{(s)}$-density point.
	\end{enumerate}

	\begin{xmpl} \label{e7}
		Let us consider the ideal $\mathcal{J}$ a subcollection of $2^{\mathbb{N}}$ where $\mathcal{J}$ consists of all those subsets of $\mathbb{N}$ whose asymptotic density is zero. Let us take the set $A$ as the open interval $(-1,1)$ and the point $p$ to be $0$. For any positive real sequence $\{s_n\}_{n \in \mathbb{N}}$ let us consider a collection of closed bounded intervals $J_n=\left[-\frac{1}{s_n},\frac{1}{s_n}\right]$ for all $n \in \mathbb{N}$. We make a choice of the sequence $\{s_n\}_{n \in \mathbb{N}}$ as follows:
		
		\begin{equation*}
			s_n =
			\left\{
			\begin{array}{ll}
				n!  & \mbox{if } n \neq m^2 \ \mbox{where} \ m \in \mathbb{N}\\
				n^{-1} & \mbox{if } n=m^2 \ \mbox{where} \ m \in \mathbb{N}.
			\end{array}
			\right.
		\end{equation*}
		
		Then $s_2<s_3<s_5<s_6<s_7< \dots$ and also the set of positive integers $\{2,3,5,6,7,8,10,\dots\}=\{n \in \mathbb{N}:n \neq m^2 \ \mbox{where} \ m \in \mathbb{N}\} \in \mathcal{F}(\mathcal{J})$, since $\{n \in \mathbb{N}:n \neq m^2 \ \mbox{where} \ m \in \mathbb{N}\}$ has natural density zero. So, $\{s_n\}_{n \in \mathbb{N}}$ is $\mathcal{J}$-monotonic increasing and unbounded positive real sequence. So, $\{s_n\} \in \Sigma_{\mathcal{J}}$.
		
		Now the sequence $x_n=\frac{\lambda(A \cap J_n)}{|J_n|}$ for $n \in \mathbb{N}$ becomes
		
		\begin{equation*}
			x_n =
			\left\{
			\begin{array}{ll}
				1  & \mbox{if } n \neq m^2 \ \mbox{where} \ m \in \mathbb{N}\\
				n^{-1} & \mbox{if } n=m^2 \ \mbox{where} \ m \in \mathbb{N}.
			\end{array}
			\right.
		\end{equation*}
		Therefore, since the subsequence $\{x_{n}\}_{n=m^2}$ converges to $0$ and the subsequence $\{x_n\}_{n \neq m^2 }$ converges to $1$, $\lim_{n} x_n$ does not exists. Since for any $\varepsilon >0$, $\{n:|x_n-1| \geq \varepsilon \} \subseteq \{n:n = m^2 \ \mbox{where} \ m \in \mathbb{N}\}$ and $\{n:n = m^2 \ \mbox{where} \ m \in \mathbb{N}\} \in \mathcal{J}$, so $\{n:|x_n-1| \geq \varepsilon \} \in \mathcal{J}$. Thus, $\mathcal{J}-\lim_{n} x_n=1$. Consequently, 0 is a $\mathcal{J}_{(s)}$-density point of $A$.

	\end{xmpl}
	
	In order to establish that $\mathcal{I}_{(s)}$-density point is indeed a generalization of $\langle s \rangle$-density point we need the following theorem and its corollary.
	
	\begin{thm}
		Let $A \in \mathcal{L}$, $x \in \mathbb{R}$ and $\langle s \rangle = \{s_n\}_{n \in \mathbb{N}} \in \mathcal{S}$. If $\{r_n\}_{n \in \mathbb{N}}$ be any real sequence such that there exists $n_0 \in \mathbb{N}$, for which $s_n=r_n$  for $n \geq n_0$, then $x$ is $\langle s \rangle -$density point of $A$ if and only if $x$ is $\{r_n\}-$density point of $A$.
	\end{thm}
	
	\begin{proof}
		Since $s_n=r_n$  for $n \geq n_0$ so
		$$\frac{\lambda \left(A \cap \left[x-\frac{1}{s_n},x+\frac{1}{s_n}\right]\right)}{\frac{2}{s_n}}=\frac{\lambda \left(A \cap \left[x-\frac{1}{r_n},x+\frac{1}{r_n}\right]\right)}{\frac{2}{r_n}} \quad \forall n\geq n_0.$$ 
		Thus, 
		\begin{align*}
			x \ \mbox{is} \ \langle s \rangle-\mbox{density point of} \ A &\Leftrightarrow \lim_{n \rightarrow{\infty}} \frac{\lambda \left(A \cap \left[x-\frac{1}{s_n},x+\frac{1}{s_n}\right]\right)}{\frac{2}{s_n}}=1\\
			&\Leftrightarrow \lim_{n \rightarrow{\infty}}\frac{\lambda \left(A \cap \left[x-\frac{1}{r_n},x+\frac{1}{r_n}\right]\right)}{\frac{2}{r_n}}=1\\
			&\Leftrightarrow x \ \mbox{is} \ \{r_n\}-\mbox{density point of} \ A.
		\end{align*}
		This completes the proof of the theorem.
	\end{proof}
	
	\begin{crlre} \label{e8}
		In Definition \ref{e3} if we choose $\{s_n\}$ to be any unbounded positive real sequence and there exists $n_0 \in \mathbb{N}$ for which $\{s_n\}$ for $n \geq n_0$ is non-decreasing, then the definition remains valid as well and we can write without any loss of generality, $x$ is $\langle s \rangle -$density point of $A$.
	\end{crlre}
	
	\begin{note}
		In particular in Definition \ref{e6} if we take $\mathcal{I}= \mathcal{I}_{fin}$ where $\mathcal{I}_{fin}$ is the class of all finite subsets of $\mathbb{N}$, then the collection $\Sigma_{\mathcal{I}}$ contains unbounded and $\mathcal{I}_{fin}$-monotonic increasing positive real sequences. 
	\end{note}
	
	\begin{thm}\label{fin}
		Let $A \in \mathcal{L}$, $x \in \mathbb{R}$ and $(s)= \{s_n\}_{n \in \mathbb{N}}\in \Sigma_{\mathcal{I}_{fin}}$. If $x$ is an $\mathcal{I}_{fin(s)}-$density point of $A$ then it is $\langle s \rangle -$density point of $A$.
	\end{thm}

	\begin{proof}
		Let $(s) \in \Sigma_{\mathcal{I}_{fin}}$. So, $(s)$ is $\mathcal{I}_{fin}$-monotonic increasing positive real sequence. Thus, there exists $\{k_1<k_2<k_3< \cdots\} \in \mathcal{F}(\mathcal{I}_{fin})$ such that $s_{k_i} \leq s_{k_{i+1}}$ for every $i \in \mathbb{N}$. Now if $\mathbb{N} \setminus \{k_1<k_2<k_3< \cdots\} =\{l_1,l_2, \cdots, l_{N}\}$, then there exists some $r_0 \in \mathbb{N}$ such that $l_{N}<k_{r_0}$. So, $\{s_n\}_{n \geq k_{r_0}}$ is non-decreasing. We claim that for $t \in \mathbb{N}$,
		\begin{equation}\label{consecutive}
			k_{r_0+(t+1)}=k_{(r_0+t)}+1.
		\end{equation}
		Since, $\{k_1<k_2<k_3< \cdots\}\cup \{l_1,l_2, \cdots, l_{N}\}=\mathbb{N}$ and $l_{N}<k_{r_0}$ so $k_{r_0+i}=k_{r_0}+i$ for $i \in \mathbb{N}$. Thus we get consecutive natural numbers from $k_{r_0}$ onwards in the set $\{k_1<k_2<k_3< \cdots\}$.

		Now let $x$ is $\mathcal{I}_{fin(s)}-$density point of $A$. So for given any $\epsilon >0$,
		$$\left\{n \in \mathbb{N}:1-\epsilon <\frac{\lambda \left(A \cap \left[x-\frac{1}{s_n},x+\frac{1}{s_n}\right]\right)}{\frac{2}{s_n}}<1+\epsilon \right\} \in \mathcal{F}(\mathcal{I}_{fin}).$$
		Thus, 
		$$\mathcal{B}=\left\{n \in \mathbb{N}:1-\epsilon <\frac{\lambda \left(A \cap \left[x-\frac{1}{s_n},x+\frac{1}{s_n}\right]\right)}{\frac{2}{s_n}}<1+\epsilon \right\} \cap  \{k_1<k_2<k_3< \cdots\} \in \mathcal{F}(\mathcal{I}_{fin}).$$
		As a result, $\mathcal{B} \subset \{k_1<k_2<k_3< \cdots\}$ and so $\{k_1<k_2<k_3< \cdots\}=\mathcal{B} \cup \mathcal{C}$ for some set $\mathcal{C}$. Clearly, $$\mathcal{C}=\{k_1<k_2<k_3< \cdots\} \setminus \mathcal{B}=\{k_1<k_2<k_3< \cdots\} \cap \mathcal{B}^{c}.$$ 
		Thus, $\mathcal{C} \subset \mathcal{B}^{c}$ and $\mathcal{B} \in \mathcal{F}(\mathcal{I}_{fin})$, which implies $\mathcal{C} \in \mathcal{I}_{fin}$. So, $\mathcal{C}$ is a finite set. Consequently, $$\mathcal{B}=\{k_1<k_2<k_3< \cdots\} \setminus \mathcal{C}.$$
		
		Clearly, for any given $\epsilon>0$, 
		$$\left\{n \in \mathbb{N}:1-\epsilon <\frac{\lambda \left(A \cap \left[x-\frac{1}{s_n},x+\frac{1}{s_n}\right]\right)}{\frac{2}{s_n}}<1+\epsilon \right\} \supset \{k_1<k_2<k_3< \cdots\} \setminus \mathcal{C} \ \mbox{where} \ \mathcal{C} \  \mbox{is a finite set}.$$
		So there exists $q_0 \in \mathbb{N}$ such that 
		$$\left\{n \in \mathbb{N}:1-\epsilon <\frac{\lambda \left(A \cap \left[x-\frac{1}{s_n},x+\frac{1}{s_n}\right]\right)}{\frac{2}{s_n}}<1+\epsilon \right\} \supset \{k_{q_0}<k_{q_0+1}<k_{q_0+2}< \cdots\}.$$
		In particular if we choose $k_{m_0}=\max \{{k_{r_0},k_{q_0}}\}$, then $s_{k_{m_0}}<s_{k_{m_0+1}}<s_{k_{m_0+2}}< \cdots$ and we note that by \ref{consecutive}, $k_{m_0+(t+1)}=k_{m_0+t}+1$ for $t \in \mathbb{N} \cup \{0\}$.

		Consequently, for all $n \geq k_{m_0}$ the set $\left\{n \in \mathbb{N}:1-\epsilon <\frac{\lambda \left(A \cap \left[x-\frac{1}{s_n},x+\frac{1}{s_n}\right]\right)}{\frac{2}{s_n}}<1+\epsilon \right\}$ contains all consecutive natural numbers on and from $k_{m_0}$. i.e.
		$$1-\epsilon <\frac{\lambda \left(A \cap \left[x-\frac{1}{s_n},x+\frac{1}{s_n}\right]\right)}{\frac{2}{s_n}}<1+\epsilon \ \forall n \geq k_{m_0}.$$
		Thus, $\lim_{n \rightarrow{\infty}} \frac{\lambda \left(A \cap \left[x-\frac{1}{s_n},x+\frac{1}{s_n}\right]\right)}{\frac{2}{s_n}}=1$. Now, by Corollary \ref{e8} we can conclude that $x$ is $\langle s \rangle -$density point of $A$. As a result our definition of $\mathcal{I}_{(s)}$-density coincides with the definition of $\langle s \rangle$-density when $\mathcal{I}= \mathcal{I}_{fin}$.
	\end{proof}

	In the next example we investigate the role played by sequences in Definition \ref{e6}. We define the set $-A=\{-x: x \in A\}.$
	
	\begin{xmpl}\label{e9}
		This example gives some insight to the role a sequence plays in the above case.
	\end{xmpl}
	
	For the ideal $\mathcal{J}$ as in Example \ref{e7} if we make a choice of the sequence $\{s_n\}_{n \in \mathbb{N}}$ as following:
	
	\begin{equation*}
		s_n =
		\left\{
		\begin{array}{ll}
			n!  & \mbox{if } n \neq m^2 \ \mbox{where} \ m \in \mathbb{N}\\
			n^{-1} & \mbox{if } n=m^2 \ \mbox{where} \ m \in \mathbb{N}.
		\end{array}
		\right.
	\end{equation*}
	
	Then clearly from Example \ref{e7} $\{s_n\} \in \Sigma_{\mathcal{J}}$ and it is denoted by $(s)$. Now let us take a set 
	
	\begin{equation*}
		A=\bigcup_{n=1}^{\infty}\left[\frac{1}{(n+1)!},\frac{1}{n!\sqrt{n+1}}\right].
	\end{equation*}
	Then $A$ is Lebesgue measurable as it is countable union of closed intervals. Now since $A$ is a subset of $[0,1]$ so $\lambda(A)=\omega \leq 1$ for some positive real number $\omega$. We define $x_n =s_n \lambda \left(A \cap \left[0,\frac{1}{s_n}\right]\right)$ for $n \in \mathbb{N}$.\\
	Now if $n \neq m^2$, then
	\begin{equation*}
		x_n= s_n \lambda \left(A \cap \left[0,\frac{1}{s_n}\right]\right) = n!\lambda 
		\left(A \cap \left[0,\frac{1}{n!}\right]\right) \leq \frac{n!}{n! \sqrt{n+1}}=\frac{1}{\sqrt{n+1}}.
	\end{equation*}
	If $n=m^2$, then
	
	\begin{equation*}
		x_n= s_n \lambda \left(A \cap \left[0,\frac{1}{s_n}\right]\right) = \frac{\lambda \left(A \cap \left[0,n \right] \right)}{n}=\frac{\lambda(A)}{n}=\frac{\omega}{n}\leq \frac{1}{n}.
	\end{equation*}
	
	Since $\{x_n\}$ is a sequence of non-negative real numbers, so $\lim_{n} x_n =0$. Hence, $\mathcal{J}-\lim_{n} x_n =0$. So, $0$ is a right $\mathcal{J}_{(s)}$-dispersion point of $A$ for $(s)$ in $\Sigma_{\mathcal{J}}$. Now if we take $-A=\bigcup_{n=1}^{\infty}\left[-\frac{1}{(n+1)!},-\frac{1}{n!\sqrt{n+1}}\right]$, then by similar calculation it can be shown that $0$ is a left $\mathcal{J}_{(s)}$-dispersion point of $-A$ for $(s) \in \Sigma_{\mathcal{J}}$. We observe that $-A \cup A$ is symmetric about origin. Clearly, $0$ is both right and left $\mathcal{J}_{(s)}$-dispersion point of $-A \cup A$ for $(s)$ in $\Sigma_{\mathcal{J}}$. Consequently, $0$ is a $\mathcal{J}_{(s)}$-dispersion point of $-A \cup A$ for $(s)$ in $\Sigma_{\mathcal{J}}$.\\
	
	Now, instead of taking the sequence $\{s_n\}_{n \in \N}$ we make some other choice of sequence $\{c_n\}_{n \in \N}$ where
	
	\begin{equation*}
		c_n =
		\left\{
		\begin{array}{ll}
			n! \sqrt{n+1}  & \mbox{if } n \neq m^2 \ \mbox{where} \ m \in \mathbb{N}\\
			n^{-1} & \mbox{if } n=m^2 \ \mbox{where} \ m \in \mathbb{N}.
		\end{array}
		\right.
	\end{equation*}
	
	Then, $c_2<c_3<c_5<c_6<c_7< \dots$ and so the set of positive integers $\{2,3,5,6,7,8,10,\dots\}=\{n \in \mathbb{N}:n \neq m^2 \ \mbox{where} \ m \in \mathbb{N}\} \in \mathcal{F}(\mathcal{J})$ i.e. $\{c_n\}_{n \in \mathbb{N}}$ is $\mathcal{J}$-monotonic increasing and unbounded positive real sequence. So, $\{c_n\} \in \Sigma_{\mathcal{J}}$ and we denote $\{c_n\}$ as $(c)$. We define $y_n =c_n \lambda \left(A \cap \left[0,\frac{1}{c_n}\right]\right)$ for all $n \in \mathbb{N}$. Thus, for $n \neq m^2$,
	
	\begin{equation*}
		\begin{split}
			y_n &= n!\sqrt{n+1} \ \lambda\left(A \cap \left[0,\frac{1}{n!\sqrt{n+1}}\right]\right)\\
			&= n!\sqrt{n+1} \sum_{k=n}^{\infty} \left( \frac{1}{k!\sqrt{k+1}}-\frac{1}{(k+1)!}\right)\\
			&= n!\sqrt{n+1} \lim_{p \rightarrow{\infty}} \sum_{k=n}^{n+p} \left( \frac{1}{k!\sqrt{k+1}}-\frac{1}{(k+1)!}\right)\\
			&= \lim_{p \rightarrow{\infty}} n!\sqrt{n+1}  \sum_{k=n}^{n+p} \left( \frac{1}{k!\sqrt{k+1}}-\frac{1}{(k+1)!}\right)\\
			&= \lim_{p \rightarrow{\infty}} S_p
		\end{split}
	\end{equation*}
	
	where
	\begin{equation*}
		\begin{split}
			S_p &= n!\sqrt{n+1} \   \sum_{k=n}^{n+p} \left( \frac{1}{k!\sqrt{k+1}}-\frac{1}{(k+1)!}\right)\\
			&= \left(1-\frac{1}{\sqrt{n+1}}\right)+\left(\frac{1}{\sqrt{n+1}\sqrt{n+2}}-\frac{1}{(n+2)\sqrt{n+1}}\right)+ \left(\frac{1}{(n+2) \sqrt{n+1}\sqrt{n+3}} \right.\\
			&\left. -\frac{1}{(n+3)(n+2) \sqrt{n+1}} \right)+\dots + \left(\frac{1}{(n+p)(n+p-1)\dots (n+2)\sqrt{n+1}\sqrt{n+p+1}} \right.\\
			&\left. -\frac{1}{(n+p+1)(n+p) \dots (n+2) \sqrt{n+1}}\right)\\
			&\geq \left(1-\frac{1}{\sqrt{n+1}}\right)+\left(\frac{1}{\sqrt{n+1}\sqrt{n+2}}-\frac{1}{\sqrt{n+2}\sqrt{n+1}}\right)
			+ \left(\frac{1}{(n+2) \sqrt{n+1}\sqrt{n+3}} \right.\\
			&\left. -\frac{1}{(n+2) \sqrt{n+1} \sqrt{n+3}}\right)+\dots + \left(\frac{1}{(n+p)(n+p-1)\dots (n+2)\sqrt{n+1}\sqrt{n+p+1}} \right. \\
			&\left. -\frac{1}{(n+p) \dots (n+2) \sqrt{n+1}\sqrt{n+p+1}}\right)\\
			&= 1- \frac{1}{\sqrt{n+1}}.
		\end{split}
	\end{equation*}

	Note that the final term in the R.H.S. is independent of $p$. So,
	\begin{equation*}
		y_n \geq 1-\frac{1}{\sqrt{n+1}} \  \text{for all} \ n\neq m^2 \ \mbox{where} \ m \in \mathbb{N}.
	\end{equation*}
	
	For, $n=m^2 \ \mbox{where} \ m \in \mathbb{N}$, 
	\begin{equation*}
		y_n= c_n \lambda \left(A \cap \left[0,\frac{1}{c_n}\right]\right) = \frac{\lambda \left(A \cap \left[0,n \right] \right)}{n}=\frac{\lambda(A)}{n}=\frac{\omega}{n}\leq \frac{1}{n}.
	\end{equation*}
	
	So, clearly, 
	\begin{equation*}
		1-\frac{1}{\sqrt{n+1}} \leq y_n \leq 1+\frac{1}{\sqrt{n+1}} \  \text{for all} \ n\neq m^2 \ \mbox{where} \ m \in \mathbb{N}.
	\end{equation*}
	
	Thus, $\left\{n: |y_n-1|\leq \frac{1}{\sqrt{n+1}}\right\} \supseteq \{n: n \neq m^2 \ \mbox{where} \ m \in \mathbb{N}\}$. For arbitrary small $\epsilon>0$, choose $n$ large enough, say there exists $n_0 \in \mathbb{N}$ so that for $n > n_0$ we have $\frac{1}{\sqrt{n+1}}< \epsilon$. Then, $$\left\{n: |y_n-1|< \epsilon\right\} \supseteq \{n: n \neq m^2 \ \mbox{where} \ m \in \mathbb{N}\} \setminus \{1,2, \cdots, n_0\}.$$ 
	Since, $\{n: n \neq m^2 \ \mbox{where} \ m \in \mathbb{N}\} \setminus \{1,2, \cdots, n_0\} \in \mathcal{F}(\mathcal{J})$ so, $\left\{n: |y_n-1|< \epsilon\right\} \in \mathcal{F}(\mathcal{J})$ and thus $\mathcal{J}-\lim_{n \rightarrow{\infty}} y_n=1$. Consequently, $0$ is right $\mathcal{J}_{(c)}$-density point of $A$. By similar calculation it can be shown that $0$ is a left $\mathcal{J}_{(c)}$-density point of $-A$ for $(c)$ in $\Sigma_{\mathcal{J}}$. Clearly, $0$ is both right and left $\mathcal{J}_{(c)}$-density point of $-A \cup A$ for $(c)$ in $\Sigma_{\mathcal{J}}$. Consequently, $0$ is a $\mathcal{J}_{(c)}$-density point of $-A \cup A$ for $(c)$ in $\Sigma_{\mathcal{J}}$. Although $0$ is $\mathcal{J}_{(s)}$-dispersion point of $-A \cup A$ for $(s)$ in $\Sigma_{\mathcal{J}}$.

	\begin{rmrk}
		Thus in general, for any given set $A \subset \mathbb{R}$, the notion of $\mathcal{I}_{(s)}$-density point of $A$ with respect to the sequence $(s) \in \Sigma_{\mathcal{I}}$ is dependent on the nature of the sequence $(s)$. It may vary from sequence to sequence.
	\end{rmrk} 
	
	\section{$\mathcal{I}_{(s)}$-density topology}\label{section 3}

	For any real sequence $(s) \in \Sigma_{\mathcal{I}}$ and any set $A \in \mathcal{L}$ let us consider the collection
	
	\begin{equation*}
		\Phi^{\mathcal{I}}_{(s)}(A)=\{x \in \mathbb{R}: x\  \mbox{is} \     \mathcal{I}_{(s)}- \mbox{density point of A} \}.
	\end{equation*}
	
	\begin{note}\label{note}
		If $\mathcal{I}=\mathcal{I}_{fin}$, then by Theorem \ref{fin} 
		\begin{align*}
			\Phi^{\mathcal{I}_{fin}}_{(s)}(A) &= \{x \in \mathbb{R}: x\  \mbox{is} \     \mathcal{I}_{fin(s)}- \mbox{density point of A} \}\\
			&= \{x \in \mathbb{R}: x\  \mbox{is} \     \langle s \rangle- \mbox{density point of A} \}\\
			&= \Phi_{\langle s \rangle}(A)
		\end{align*}
	\end{note}
	
	\begin{dfn} \cite{dasgupta}
		We recall that countable union of closed sets is called $F_{\sigma}$ sets. Countable intersection of $F_{\sigma}$ sets is called $F_{\sigma \delta}$. Thus, $F_{\sigma \delta}:= (F_{\sigma})_{\delta}$.
	\end{dfn}
	
	\begin{lma} \label{e10} If $H$ and $G$ are any two Lebesgue measurable sets then $|\lambda(H)-\lambda(G)| \leq \lambda(H \triangle G)$.
		
	\end{lma}
	
	\begin{proof}
		If $H \cap G =\emptyset$, then $H \setminus G =H$ and $G \setminus H =G$. So, 
		\begin{align*}
			|\lambda(H)-\lambda(G)| &\leq |\lambda(H)|+ |\lambda(G)|\\
			&= \lambda(H)+ \lambda(G)\\
			&= \lambda(H \setminus G)+ \lambda(G \setminus H)\\
			&= \lambda((H \setminus G) \cup (G \setminus H)) \ \mbox{since} \ \lambda \ \mbox{is countably additive}\\
			&= \lambda(H \triangle G).
		\end{align*}
		If $H \cap G  \neq \emptyset$, then $H=(H \setminus G) \cup (H \cap G)$ and $G=(G \setminus H) \cup (G \cap H)$. So,
		\begin{align*}
			|\lambda(H)-\lambda(G)| &=|\lambda(H \setminus G)+\lambda(H \cap G)-\lambda(G \setminus H)-\lambda(G \cap H)|\\
			&=|\lambda(H \setminus G)-\lambda(G \setminus H)|\\
			&\leq |\lambda(H \setminus G)|+ |\lambda(G \setminus H)|\\
			&= \lambda(H \setminus G)+ \lambda(G \setminus H)\\
			&= \lambda((H \setminus G) \cup (G \setminus H)) \ \mbox{since} \ \lambda \ \mbox{is countably additive}\\
			&= \lambda(H \triangle G).
		\end{align*}
		Thus in the above two cases the result holds good. This completes the proof.
	\end{proof}

	\begin{ppsn}
		For any Lebesgue measurable set $A \in \mathcal{L}$ and a sequence $(s) \in \Sigma_{\mathcal{I}}$ the set $\Phi^{\mathcal{I}}_{(s)}(A)$ is a $F_{\sigma \delta}$ set.
	\end{ppsn}
	
	\begin{proof}
		For $A \in \mathcal{L}$ and $(s) \in \Sigma_{\mathcal{I}}$, let us consider the function $G(p,n): \mathbb{R} \times \N \rightarrow{\mathbb{R}}$ defined as $$G(p,n)=\lambda \left(A \cap \left[p-\frac{1}{s_n},p+\frac{1}{s_n}\right]\right)$$
		Now for $p,q \in \mathbb{R}$ and fixed $n \in \N$ we get by Lemma \ref{e10},
		\begin{equation*}
			\begin{split}
				|G(p,n)-G(q,n)| &= \left|\lambda \left(A \cap \left[p-\frac{1}{s_n},p+\frac{1}{s_n}\right]\right)-\lambda \left(A \cap \left[q-\frac{1}{s_n},q+\frac{1}{s_n}\right]\right)\right|\\
				&\leq \lambda \left(\left(A \cap \left[p-\frac{1}{s_n},p+\frac{1}{s_n}\right]\right) \triangle \left(A \cap \left[q-\frac{1}{s_n},q+\frac{1}{s_n}\right]\right)\right)\\
				&= \lambda \left(A \cap \left( \left[p-\frac{1}{s_n},p+\frac{1}{s_n}\right] \triangle  \left[q-\frac{1}{s_n},q+\frac{1}{s_n}\right]\right)\right)\\
				&\leq \left|\left[p-\frac{1}{s_n},p+\frac{1}{s_n}\right] \triangle  \left[q-\frac{1}{s_n},q+\frac{1}{s_n}\right]\right|\\
				&\leq 2|p-q|.
			\end{split}
		\end{equation*}
		Hence $G(.,n)$ for fixed $n$ satisfies Lipschitz condition. So it is continuous. So for fixed $n$ the function $\frac{s_n}{2} G(p,n)$ is continuous with respect to $p$. Now, $p \in \Phi^{\mathcal{I}}_{(s)}(A)$ if and only if for any $F_k=\{k_1<k_2< \dots\}\in \mathcal{F}(\mathcal{I})$ such that $s_{k_i} \leq s_{k_{i+1}} \ \forall i \in \mathbb{N}$ we have for each $r \in \mathbb{N}$ there exists $m \in \mathbb{N}$ such that for each $n>m$ and $n \in F_k$, 
		$$\frac{s_n}{2}   G(p,n) \geq 1-\frac{1}{r}.$$

		Hence, 
		$$\Phi^{\mathcal{I}}_{(s)}(A)= \bigcap_{r=1}^{\infty} \bigcup_{m \in \N} \bigcap _{n > m, n \in F_k} \left\{p \in \mathbb{R}: \frac{s_n}{2} G(p,n) \geq 1- \frac{1}{r} \right\}.$$ 
		Since, $\frac{s_n}{2} G(p,n)$ is continuous with respect to $p$ so $\left\{p \in \mathbb{R}: \frac{s_n}{2} G(p,n) \geq  1- \frac{1}{r}\right\}$ is a closed set. Therefore, $\Phi^{\mathcal{I}}_{(s)}(A) \in F_{\sigma \delta}$. In particular $\Phi^{\mathcal{I}}_{(s)}(A) \in \mathcal{L}$.
	\end{proof}
	
	\begin{ppsn}\label{e11}
		For every pair of Lebesgue measurable sets $A,B \in \mathcal{L}$ and a sequence $(s) \in \Sigma_{\mathcal{I}}$ we have
	\end{ppsn}
	\begin{enumerate}
		\item $\Phi^{\mathcal{I}}_{(s)}(\emptyset)=\emptyset$ , $\Phi^{\mathcal{I}}_{(s)}(\mathbb{R})=\mathbb{R}$;
		\item $\Phi^{\mathcal{I}}_{(s)}(A \cap B)=\Phi^{\mathcal{I}}_{(s)}(A) \cap \Phi^{\mathcal{I}}_{(s)}(B)$;
		\item $A \sim B \implies \Phi^{\mathcal{I}}_{(s)}(A)=\Phi^{\mathcal{I}}_{(s)}(B)$;
		\item $\Phi(A) \subseteq \Phi_{\langle s \rangle}(A) \subseteq \Phi^{\mathcal{I}}_{(s)}(A)$;
		\item $\Phi^{\mathcal{I}}_{(s)}(A) \sim A$.
		
	\end{enumerate}
	
	\begin{proof}
		\begin{enumerate}
			\item $\Phi^{\mathcal{I}}_{(s)}(\emptyset)=\emptyset$ by voidness since an empty set has no points so it has no $\mathcal{I}_{(s)}$-density points.

			Clearly, $\Phi^{\mathcal{I}}_{(s)}(\mathbb{R}) \subseteq \mathbb{R}$. Now, for any $x \in \mathbb{R}$ let $J_n=\left[x-\frac{1}{s_n},x+\frac{1}{s_n}\right]$ for all $n \in \mathbb{N}$. Then
			$$\frac{\lambda \left(\mathbb{R} \cap J_n\right)}{|J_n|}=\frac{\lambda (J_n)}{|J_n|}=\frac{|J_n|}{|J_n|}=1 \ \text{for all} \ n \in \N.$$
			So for given any $\epsilon>0$, $\left\{n \in \mathbb{N}: \left|\frac{\lambda \left(\mathbb{R} \cap J_n\right)}{|J_n|}-1\right|<\epsilon\right\}=\mathbb{N}\in \mathcal{F}(\mathcal{I})$. Thus $\mathcal{I}-\lim_{n \rightarrow{\infty}} \frac{\lambda \left(\mathbb{R} \cap J_n\right)}{|J_n|}=1$. So, $x \in \Phi^{\mathcal{I}}_{(s)}(\mathbb{R})$. Hence, $\Phi^{\mathcal{I}}_{(s)}(\mathbb{R})=\mathbb{R}$.
			
			\item Since $A \cap B \subseteq A$ and $A \cap B \subseteq B$, so $\Phi^{\mathcal{I}}_{(s)}(A \cap B) \subseteq \Phi^{\mathcal{I}}_{(s)}(A)$ and $\Phi^{\mathcal{I}}_{(s)}(A \cap B) \subseteq \Phi^{\mathcal{I}}_{(s)}(B)$. Consequently, $\Phi^{\mathcal{I}}_{(s)}(A \cap B) \subseteq \Phi^{\mathcal{I}}_{(s)}(A) \cap \Phi^{\mathcal{I}}_{(s)}(B)$. Now we are to prove $ \Phi^{\mathcal{I}}_{(s)}(A) \cap \Phi^{\mathcal{I}}_{(s)}(B) \subseteq \Phi^{\mathcal{I}}_{(s)}(A \cap B)$. Let $x \in \Phi^{\mathcal{I}}_{(s)}(A) \cap \Phi^{\mathcal{I}}_{(s)}(B)$. Thus $x \in \Phi^{\mathcal{I}}_{(s)}(A)$ and $x \in  \Phi^{\mathcal{I}}_{(s)}(B)$. For $J_n=\left[x-\frac{1}{s_n},x+\frac{1}{s_n}\right] \ \forall n \in \mathbb{N}$ and for given any $\epsilon >0$ we have 
			$$A_{\epsilon}=\left\{n:\frac{\lambda(A \cap  J_n)}{|J_n|}>1-\epsilon\right\} \in \mathcal{F}(\mathcal{I}) \ \mbox{and} \ B_{\epsilon}=\left\{n:\frac{\lambda(B \cap  J_n)}{|J_n|}>1-\epsilon\right\} \in \mathcal{F}(\mathcal{I}).$$
			Now since,      
			$$\lambda(A \cap J_n)+\lambda(B \cap J_n)-\lambda(A \cap B \cap J_n) \leq |J_n|$$
			
			so for any $\{k_1<k_2< \dots\}\in \mathcal{F}(\mathcal{I})$ such that $s_{k_i} \leq s_{k_{i+1}} \ \forall i \in \mathbb{N}$ we have for $n \in \{k_1<k_2< \dots\}$,
			\begin{equation}\label{eq1}
				\frac{\lambda(A \cap  J_n)}{|J_n|}+\frac{\lambda(B \cap  J_n)}{|J_n|} \leq 1+ \frac{\lambda((A \cap B)\cap J_n)}{|J_n|}.
			\end{equation}

			So for $n \in \{k_1<k_2< \dots\} \cap A_{\epsilon} \cap B_{\epsilon}$ from equation \ref{eq1} we have
			\begin{align*}
				\frac{\lambda((A \cap B)\cap J_n)}{|J_n|} & \geq \frac{\lambda(A \cap J_n)}{|J_n|}+ \frac{\lambda(B \cap J_n)}{|J_n|}-1\\
				&> 1-2\epsilon.
			\end{align*}
			Thus, $\left\{n:\frac{\lambda((A \cap B)\cap J_n)}{|J_n|}>1-2\epsilon\right\} \supseteq \{k_1<k_2< \dots\} \cap A_{\epsilon} \cap B_{\epsilon}$ and $\{k_1<k_2< \dots\} \cap A_{\epsilon} \cap B_{\epsilon} \in \mathcal{F}(\mathcal{I})$. So, $\mathcal{I}-\lim_n \frac{\lambda((A \cap B) \cap  J_n)}{|J_n|}= 1$. Therefore, $x \in \Phi^{\mathcal{I}}_{(s)}(A \cap B)$. So we are done. As a corollary to this we can conclude $\Phi^{\mathcal{I}}_{(s)}(A) \subseteq \Phi^{\mathcal{I}}_{(s)}(B)$ for $A \subseteq B$ i.e. $\Phi^{\mathcal{I}}_{(s)}(.)$ is monotonic. 
			
			\item Let $\{J_n\}_{n \in \mathbb{N}}$ be any sequence of closed interval in $\mathbb{R}$. If $\lambda(A \triangle B)=0$ then we claim that $\lambda(A \cap J_n)=\lambda(B \cap J_n)$ for each interval $J_n \subset \mathbb{R}$. Now
			\begin{equation*}
				\begin{split}
					A &=A \cap (B \cup B^{c})\\
					&= (A \cap B) \cup (A \cap B^{c})\\
					&= (A \cap B) \cup (A \setminus B)\\
					&\subset B \cup (A \triangle B).
				\end{split}
			\end{equation*}
			
			So, for any $n \in \mathbb{N}$ we have
			\begin{equation*}
				\begin{split}
					\lambda(A \cap J_n) & \leq \lambda ((B \cup (A \triangle B)) \cap J_n)\\
					& \leq \lambda (B \cap J_n)+\lambda ((A \triangle B) \cap J_n) \\
					&= \lambda(B \cap J_n) \quad \mbox{since} \ \lambda ((A \triangle B) \cap J_n) \leq \lambda(A \triangle B)=0.
				\end{split}
			\end{equation*}
			Similarly, $\lambda (B \cap J_n) \leq \lambda(A \cap J_n)$ for all $n \in \mathbb{N}$. So, we have $\lambda (A \cap J_n) = \lambda(B \cap J_n)$ for all $n \in\mathbb{N}$. For $(s) \in \Sigma_{\mathcal{I}}$ let $J_n=\left[x-\frac{1}{s_n},x+\frac{1}{s_n}\right]$ for all $n \in \mathbb{N}$. Then,
			
			\begin{align*}
				x \in \Phi^{\mathcal{I}}_{(s)}(A) &\Leftrightarrow \mathcal{I}-\lim_n \frac{\lambda(A \cap  J_n)}{|J_n|}= 1\\
				&\Leftrightarrow \mathcal{I}-\lim_n \frac{\lambda(B \cap  J_n)}{|J_n|}= 1\\
				&\Leftrightarrow x \in \Phi^{\mathcal{I}}_{(s)}(B)
			\end{align*}
			Consequently, $\Phi^{\mathcal{I}}_{(s)}(A)= \Phi^{\mathcal{I}}_{(s)}(B)$.
			
			\item By Proposition 2 from \cite{Filipczak 2004} we have $\Phi(A) \subseteq \Phi_{\langle s \rangle}(A)$. Now we claim that $\Phi_{\langle s \rangle}(A) \subseteq \Phi^{\mathcal{I}}_{(s)}(A)$.
			We notice that if $\mathcal{I}$ is an admissible ideal then $\mathcal{I}_{fin} \subset \mathcal{I}$. For any $x \in \mathbb{R}$ let $x \in \Phi_{\langle s \rangle}(A)$. Then by Note \ref{note}, $x \in \Phi^{\mathcal{I}_{fin}}_{(s)}(A)$. Thus for given any $\epsilon>0$, 
			$$\left\{n \in \mathbb{N}:\left|\frac{\lambda \left(A \cap \left[x-\frac{1}{s_n},x+\frac{1}{s_n}\right]\right)}{\frac{2}{s_n}}-1\right|\geq \epsilon \right\} \in \mathcal{I}_{fin}$$
			Thus,
			$$\left\{n \in \mathbb{N}:\left|\frac{\lambda \left(A \cap \left[x-\frac{1}{s_n},x+\frac{1}{s_n}\right]\right)}{\frac{2}{s_n}}-1\right|\geq \epsilon \right\} \in \mathcal{I} \ \mbox{since} \ \mathcal{I}_{fin} \subseteq 
			\mathcal{I}.$$
			So, $x \in \Phi^{\mathcal{I}}_{(s)}(A)$. Consequently, $\Phi_{\langle s \rangle}(A) \subset \Phi^{\mathcal{I}}_{(s)}(A)$. 
			
			\item We are to show $\lambda (\Phi^{\mathcal{I}}_{(s)}(A) \triangle A)=0$. Now, 
			$\Phi^{\mathcal{I}}_{(s)}(A) \triangle A = (A \setminus \Phi^{\mathcal{I}}_{(s)}(A)) \cup (\Phi^{\mathcal{I}}_{(s)}(A) \setminus A)$. Since $\Phi (A) \subseteq \Phi^{\mathcal{I}}_{(s)}(A)$ so $A \setminus \Phi^{\mathcal{I}}_{(s)}(A) \subseteq A \setminus \Phi (A)$. By Lebesgue density theorem \ref{e5}, $\lambda (A \setminus \Phi (A))=0$. So, $\lambda (A \setminus \Phi^{\mathcal{I}}_{(s)}(A))=0$. Now we are to show $\lambda (\Phi^{\mathcal{I}}_{(s)}(A) \setminus A)=0$. We note that $\Phi^{\mathcal{I}}_{(s)}(A) \cap \Phi^{\mathcal{I}}_{(s)}(\mathbb{R} \setminus A)=\Phi^{\mathcal{I}}_{(s)}(A \cap (\mathbb{R} \setminus A)) = \Phi^{\mathcal{I}}_{(s)}(\emptyset)=\emptyset $. Hence $\Phi^{\mathcal{I}}_{(s)}(A) \subseteq \mathbb{R} \setminus \Phi^{\mathcal{I}}_{(s)}(\mathbb{R} \setminus A)$. So, $$\Phi^{\mathcal{I}}_{(s)}(A) \setminus A \subseteq (\mathbb{R} \setminus A) \setminus \Phi^{\mathcal{I}}_{(s)}(\mathbb{R} \setminus A) \subseteq (\mathbb{R} \setminus A) \setminus \Phi(\mathbb{R} \setminus A).$$ 
			
			Since $\mathbb{R} \setminus A \in \mathcal{L}$, so by Lebesgue density theorem \ref{e5}, $\lambda ((\mathbb{R} \setminus A) \setminus \Phi(\mathbb{R} \setminus A)) =0$. Therefore, $\lambda (\Phi^{\mathcal{I}}_{(s)}(A) \setminus A)=0$ since $\lambda$ is complete measure. Hence, $\lambda (\Phi^{\mathcal{I}}_{(s)}(A) \triangle A)=0$. 
			
		\end{enumerate}
		
	\end{proof}
	
	\begin{crlre}\label{lower}
		The operator $\Phi^{\mathcal{I}}_{(s)}: \mathcal{L} \rightarrow{\mathcal{L}}$ is a lower density operator in the measure space $(\mathbb{R}, \mathcal{L}, \lambda)$.
	\end{crlre}
	
	\begin{dfn}\cite{halmos} Let $E$ be any subset of $\mathbb{R}$. Then a Lebesgue measurable set $\mathscr{G} \subseteq E$ is said to be a measurable kernel of $E$ if $\lambda ^{\star}(A)=0$, for every set $A \subseteq (E \setminus \mathscr{G})$.
		
	\end{dfn}
	
	As a consequence of Remark \ref{e1}, Theorem \ref{e2} and Corollary \ref{lower} we can have the following theorem.
	
	\begin{thm}
		For every sequence $(s) \in \Sigma_{\mathcal{I}}$ the family $\mathcal{T}^{\mathcal{I}}_{(s)}=\{A \in \mathcal{L}: A \subseteq \Phi^{\mathcal{I}}_{(s)}(A)\}$ forms a topology.  
	\end{thm}
	
	\begin{proof}
		For the sake of completeness we are giving a detailed proof here. Since, by Proposition \ref{e11} (1), $\Phi^{\mathcal{I}}_{(s)}(\emptyset)=\emptyset$ and $\Phi^{\mathcal{I}}_{(s)}(\mathbb{R})=\mathbb{R}$ and both $\emptyset$ and $\mathbb{R}$ are Lebesgue measurable, so $\mathcal{T}^{\mathcal{I}}_{(s)}$ contains $\emptyset$ and $\mathbb{R}$. Now let us take $A, B \in \mathcal{T}^{\mathcal{I}}_{(s)}$. Then $A \cap B \in \mathcal{L}$ since both $A$ and $B$ are Lebesgue measurable sets. Also, $A \cap B \subseteq A \subseteq \Phi^{\mathcal{I}}_{(s)}(A)$ and $A \cap B \subseteq B \subseteq \Phi^{\mathcal{I}}_{(s)}(B)$. As a consequence, by Proposition \ref{e11} (2) we have 
		$$A \cap B \subseteq \Phi^{\mathcal{I}}_{(s)}(A) \cap \Phi^{\mathcal{I}}_{(s)}(B)=\Phi^{\mathcal{I}}_{(s)}(A \cap B).$$
		Therefore, $A \cap B \in \mathcal{T}^{\mathcal{I}}_{(s)}$. So, $\mathcal{T}^{\mathcal{I}}_{(s)}$ is closed under finite intersection.
		
		Now, let us take any arbitrary collection of sets $\{H_t\}_{t \in \Gamma}$ in $\mathcal{T}^{\mathcal{I}}_{(s)}$, where $\Gamma$ is an arbitrary indexing set. We are to show $\bigcup_{t \in \Gamma}H_t \in \mathcal{T}^{\mathcal{I}}_{(s)}$. Let $\mathscr{G}$ be a measurable kernel of the set $\bigcup_{t \in \Gamma}H_t$. Then we claim $\mathscr{G} \cap H_t \sim H_t$ for every $t \in \Gamma$. Clearly, $\mathscr{G} \subseteq \bigcup_{t \in \Gamma}H_t$. Since, $H_t \setminus \mathscr{G} \subseteq \bigcup_{t \in \Gamma}H_t \setminus \mathscr{G}$ so, $\lambda(H_t \setminus \mathscr{G})=0$ for any $t \in \Gamma$. It can be easily verified that $H_t \setminus (\mathscr{G} \cap H_t)=H_t \setminus \mathscr{G}$ for every $t \in \Gamma$. Thus, $\lambda(H_t \setminus (\mathscr{G} \cap H_t))=0$ for every $t \in \Gamma$. Also since $\mathscr{G} \cap H_t \subseteq H_t$ so, $\lambda((\mathscr{G} \cap H_t)\setminus H_t)=0$ for every $t \in \Gamma$. Therefore, $\lambda(H_t \triangle (\mathscr{G} \cap H_t))=0$ and so by Proposition \ref{e11} (3), $\Phi^{\mathcal{I}}_{(s)}(H_t)=\Phi^{\mathcal{I}}_{(s)}(\mathscr{G} \cap H_t)$ for every $t \in \Gamma$. Thus we obtain that
		$$\mathscr{G} \subseteq \bigcup_{t \in \Gamma}H_t \subseteq \bigcup_{t \in \Gamma}\Phi^{\mathcal{I}}_{(s)}(H_t)=\bigcup_{t \in \Gamma}\Phi^{\mathcal{I}}_{(s)}(\mathscr{G} \cap H_t) \subseteq \Phi^{\mathcal{I}}_{(s)}(\mathscr{G}). $$
		Since, $\lambda$ is a complete measure and by Proposition \ref{e11} (5) $\lambda(\Phi^{\mathcal{I}}_{(s)}(\mathscr{G}) \setminus \mathscr{G})=0$, so $\bigcup_{t \in \Gamma} H_t \in \mathcal{L}$. Moreover, $$\bigcup_{t \in \Gamma}H_t \subseteq \Phi^{\mathcal{I}}_{(s)}(\mathscr{G}) \subseteq \Phi^{\mathcal{I}}_{(s)}\left(\bigcup_{t \in \Gamma}H_t\right) \ \mbox{by monotonicity of} \ \Phi^{\mathcal{I}}_{(s)}(.).$$ 
		Hence, $\bigcup_{t \in \Gamma}H_t \in \mathcal{T}^{\mathcal{I}}_{(s)}$. Consequently, $\mathcal{T}^{\mathcal{I}}_{(s)}$ is closed under arbitrary union. This completes the proof of the theorem.
	\end{proof}
	
	\begin{note}
		We call $\mathcal{T}^{\mathcal{I}}_{(s)}$ to be the $\mathcal{I}_{(s)}$-density topology on the space of reals and by Proposition \ref{e11} (4), since $\Phi(A) \subseteq \Phi_{\langle s \rangle}(A) \subseteq \Phi^{\mathcal{I}}_{(s)}(A)$ so we can conclude that $\mathcal{T}_{d} \subseteq \mathcal{T}_{\langle s \rangle} \subseteq \mathcal{T}^{\mathcal{I}}_{(s)}$.
	\end{note}

	\begin{rmrk}
		As we have introduced $\mathcal{I}_{(s)}$-density for $(s) \in \Sigma_{\mathcal{I}}$ and for $\mathcal{I}=\mathcal{I}_{fin}$, $\langle s \rangle$-density coincides with $\mathcal{I}_{(s)}$-density, so $\mathcal{T}_{\langle s \rangle} = \mathcal{T}^{\mathcal{I}}_{(s)}$ if $\mathcal{I}=\mathcal{I}_{fin}$.
	\end{rmrk}

	In the following theorem the natural properties of $\mathcal{T}^{\mathcal{I}}_{(s)}$-topologies are listed.
	
	\begin{thm}
		For any $(s)\in \Sigma_{\mathcal{I}}$ and $A \in \mathcal{T}^{\mathcal{I}}_{(s)}$ we have
	\end{thm}
	\begin{enumerate}
		\item $A+x \in \mathcal{T}^{\mathcal{I}}_{(s)} \quad \forall x \in \mathbb{R}$ where $A+x=\{a+x:a \in A\}$
		\item $-A \in \mathcal{T}^{\mathcal{I}}_{(s)}$ where $-A=\{-a:a \in A\}.$
	\end{enumerate}
	
	\begin{proof}
		\begin{enumerate}
			\item For any $(s)\in \Sigma_{\mathcal{I}}$ and $x \in \mathbb{R}$ let $A \in \mathcal{T}^{\mathcal{I}}_{(s)}$ i.e. $A \subseteq \Phi^{\mathcal{I}}_{(s)}(A)$. We are to show $A+x \subseteq \Phi^{\mathcal{I}}_{(s)}(A+x)$. For fixed $x \in \mathbb{R}$ let $b \in A+x$ which implies $b-x\in A$ and so $b-x \in \Phi^{\mathcal{I}}_{(s)}(A)$. Hence $$\mathcal{I}-\lim_{n \rightarrow{\infty}} \frac{\lambda \left(A \cap \left[b-x-
				\frac{1}{s_n},b-x+\frac{1}{s_n}\right]\right)}{\frac{2}{s_n}}=1.$$
			Now by part (c) of Theorem 2.20 \cite{Rudin} since Lebesgue measure is translation invariant so,
			
			\begin{align*}
				\lambda \left(A \cap \left[b-x-
				\frac{1}{s_n},b-x+\frac{1}{s_n}\right]\right) &=\lambda \left(x+\left(A \cap \left[b-x-
				\frac{1}{s_n},b-x+\frac{1}{s_n}\right]\right)\right)\\
				&= \lambda \left((A+x) \cap \left[b-
				\frac{1}{s_n},b+\frac{1}{s_n}\right]\right).
			\end{align*}

			Thus,
			$$\mathcal{I}-\lim_{n \rightarrow{\infty}} \frac{\lambda \left((A+x) \cap \left[b-
				\frac{1}{s_n},b+\frac{1}{s_n}\right]\right)}{\frac{2}{s_n}}=1.$$ Consequently, $b \in \Phi^{\mathcal{I}}_{(s)}(A+x)$ and so $A+x \subseteq \Phi^{\mathcal{I}}_{(s)}(A+x)$. So the result follows.
			
			\item Let $A \in \mathcal{T}^{\mathcal{I}}_{(s)}$. So, $A \subseteq \Phi^{\mathcal{I}}_{(s)}(A)$. We are to show that $-A \subseteq \Phi^{\mathcal{I}}_{(s)}(-A)$. Let $x \in -A$ so $-x \in A$. Thus $-x \in \Phi^{\mathcal{I}}_{(s)}(A)$. Hence $$\mathcal{I}-\lim_{n \rightarrow{\infty}} \frac{\lambda \left(A \cap \left[-x-
				\frac{1}{s_n},-x+\frac{1}{s_n}\right]\right)}{\frac{2}{s_n}}=1.$$
			Now by part (e) of Theorem 2.20 \cite{Rudin}, for any Lebesgue measurable subset $A$ of $\mathbb{R}$ and $k \in \mathbb{R}$, $\lambda(kA)=|k|\lambda(A)$. So,
			$$\lambda \left(A \cap \left[-x-
			\frac{1}{s_n},-x+\frac{1}{s_n}\right]\right)=\lambda \left((-A) \cap \left[x-
			\frac{1}{s_n},x+\frac{1}{s_n}\right]\right).$$
			
			Thus,    
			$$\mathcal{I}-\lim_{n \rightarrow{\infty}} \frac{\lambda \left((-A) \cap \left[x-
				\frac{1}{s_n},x+\frac{1}{s_n}\right]\right)}{\frac{2}{s_n}}=1.$$
			Consequently, $x \in \Phi^{\mathcal{I}}_{(s)}(-A)$. Therefore, $-A \subset \Phi^{\mathcal{I}}_{(s)}(-A)$. So, $-A \in \mathcal{T}^{\mathcal{I}}_{(s)}$.

		\end{enumerate}
	\end{proof}
	
	\textbf{Problem.} Is there any characterization of equality for $\mathcal{T}_d$ and $\mathcal{T}^{\mathcal{I}}_{(s)}$ as given in \cite{Filipczak 2004} for $\mathcal{T}_{d}$ and $\mathcal{T}_{\langle s \rangle}$?\\

	In the next theorem we formulate a weaker condition for the sequence $(s) \in \Sigma_{\mathcal{I}}$ so that the classical density topology coincides with $\mathcal{I}_{(s)}$-density topology.
	
	\begin{thm} \label{e12}
		
		Let $(s) \in \Sigma_{\mathcal{I}}$ be a real sequence. If for any $\{k_1<k_2< \dots <k_n< \dots \} \in \mathcal{F}(\mathcal{I})$ such that $s_{k_i} \leq s_{k_{i+1}} \ \forall i \in \mathbb{N}$, the condition 
		$\liminf \frac{s_{k_n}}{s_{k_{(n+1)}}} >0$ holds, then  $\mathcal{T}_{d} = \mathcal{T}_{(s)}^{\mathcal{I}}$. 
		
	\end{thm}
	
	\begin{proof}
		It is sufficient to show that, for any $A \in \mathcal{L}$, $\Phi(A)=\Phi^{\mathcal{I}}_{(s)}(A)$, when $(s)$ satisfies the condition given in the statement. By Proposition \ref{e11} (4) we have $\Phi(A) \subseteq \Phi^{\mathcal{I}}_{(s)}(A)$. Now, we need to show $\Phi^{\mathcal{I}}_{(s)}(A) \subseteq \Phi(A)$ i.e. if $x \in \mathbb{R}$ is an  $\mathcal{I}_{(s)}$-density point of $A$ then $x$ is classical density point of $A$. Since, $\liminf \frac{s_{k_n}}{s_{k_{(n+1)}}} >0$ so there exists a subsequence of $\{s_{k_n}\}$ say $\{s_{k_{l_n}}\}$ such that $\lim_{n \rightarrow \infty} \frac{s_{k_{l_n}}}{s_{k_{l_{n+1}}}} = \sigma >0$. Thus there exists $n_0 \in \mathbb{N}$ such that for any $n \geq n_0$ we have, $$\frac{\sigma}{2} <  \frac{s_{k_{l_n}}}{s_{k_{l_{n+1}}}} < \frac{3 \sigma}{2}.$$
		
		Since $x$ is an $\mathcal{I}_{(s)}$-density point of $A$ so clearly, 
		$$\mathcal{I}-\lim_{n \rightarrow \infty} \frac{\lambda\left(A^c \cap \left[x-\frac{1}{s_n},x+\frac{1}{s_n}\right]\right)}{\frac{2}{s_n}}=0, \ \mbox{where} \  A^c \ \mbox{denotes} \ \mathbb{R} \setminus A.$$ Thus, for any given $\epsilon >0$ the set $$C_{\epsilon}=\left\{n \in \mathbb{N}: \frac{s_n}{2}\lambda \left(A^c \cap \left[x-\frac{1}{s_n},x+\frac{1}{s_n}\right]\right)<\frac{\epsilon \sigma}{2}\right\} \in \mathcal{F}(\mathcal{I}).$$
		Now, there exists $p_0 \in \mathbb{N}$ and $p_0 > n_0$ such that for some $p \in \mathbb{N}$ such that $k_{l_{p}} \in \{k_1<k_2< \dots <k_n< \dots \} \cap C_{\epsilon}$ and $p \geq p_0$ we have $$\frac{s_{k_{l_{p}}}}{2} \lambda \left(A^c \cap \left[x-\frac{1}{s_{k_{l_{p}}}},x+\frac{1}{s_{k_{l_{p}}}}\right] \right)< \frac{\epsilon \sigma}{2}.$$
		Fix $t \in \mathbb{R}$ such that $0<t<\frac{1}{s_{k_{l_{p_0}}}}$. So, there exists $p \geq p_0$ for which $k_{l_{p}} \in \{k_1<k_2< \dots <k_n< \dots \} \cap C_{\epsilon}$ such that $\frac{1}{s_{k_{l_{p+1}}}} \leq t < \frac{1}{s_{k_{l_{p}}}}$. Hence, we have 
		\begin{equation*}
			\begin{split}
				\frac{\lambda\left(A^c \cap \left[x-t,x+t\right]\right)}{2t}  &\leq \frac{\lambda\left(A^c \cap \left[x-\frac{1}{s_{k_{l_{p}}}},x+\frac{1}{s_{k_{l_{p}}}}\right]\right)}{\frac{2}{s_{k_{l_{p+1}}}}}\\
				&= \frac{\lambda\left(A^c \cap \left[x-\frac{1}{s_{k_{l_{p}}}},x+\frac{1}{s_{k_{l_{p}}}}\right]\right)}{\frac{2}{s_{k_{l_{p}}}}} \cdot \frac{s_{k_{l_{p+1}}}}{s_{k_{l_{p}}}}\\
				&< \frac{\epsilon \sigma}{2} \cdot \frac{2}{\sigma}= \epsilon.\\
			\end{split}
		\end{equation*}
		Therefore, $x$ is a classical density point of $A$. This completes the proof of the theorem.
	\end{proof}

	In view of Theorem \ref{e12} the following open question naturally arise.
	
	\textbf{Problem.} Does the converse of the above theorem hold?

	\section*{Acknowledgements}
	\textit{The second author is grateful to The Council of Scientific and Industrial Research (CSIR), Government of India, for his fellowship
		funding under CSIR-JRF schemes (SRF fellowship File no. 09/025(0277)/2019-EMR-I) during the tenure of preparation of this research paper. 
	}

\end{document}